\documentclass[12pt,reqno]{amsart}
\usepackage{amsmath, amsthm, amssymb, stmaryrd}
\advance \topmargin by -\headheight
\advance \topmargin by -\headsep
     
\evensidemargin \oddsidemargin
\newtheorem{theorem}{Theorem}[section]
\newtheorem{lemma}[theorem]{Lemma}

\theoremstyle{definition}

\theoremstyle{remark}

\numberwithin{equation}{section}

\newcommand{\abs}[1]{\left\lvert#1\right\rvert}
\newcommand{\sumstack}[1]{\sum\limits_{\substack{#1}}}

\newcommand{\norm}[1]{\left\|#1\right\|}

\def\M{\mathfrak{M}}
\def\m{\mathfrak{m}}

  \def\A{\mathcal{A}}

\def\grB{{\mathfrak B}}

\def\grB{{\mathfrak B}}

\def\ep{\varepsilon}

\def\de{{\,{\rm d}}}
\newcommand{\sj}[1]{\sigma_{#1}}
\newcommand{\tj}[1]{\tau_{#1}}

\begin{document}
\title[Digitally Restricted Goldbach]{Digitally Restricted Sets and the Goldbach Conjecture: An Exceptional Set Result}
\author[James Cumberbatch]{James Cumberbatch}
\address{Department of Mathematics, Purdue University, 150 N. University Street, West 
Lafayette, IN 47907-2067, USA}
\email{jcumberb@purdue.edu}
\subjclass[2020]{11P32}
\keywords{Restricted Digits, Goldbach Conjecture, Exceptional Set}
\date{}
\dedicatory{}
\begin{abstract}We show that for any set $D$ of at least two digits in a given base $b$, there exists a $\delta(D,b)>0$ such that within the set $\A$ of numbers whose digits base $b$ are exclusively from $D$, the number of even integers in $\A$ which are less than $X$ and not representable as the sum of two primes is less than $|\A(X)|^{1-\delta}$.
\end{abstract}
\maketitle

\section{Introduction}
The Goldbach Conjecture states that for all sufficiently large even $n$ there exists $p_1$, $p_2$ such that
\begin{equation}
p_1+p_2=n.\label{Gold}
\end{equation}
In this paper, we study how many exceptions to the Goldbach conjecture can have restricted digits. Given a base $b$ and a set of digits $D\subseteq \{0,1,2,...,b-1\}$ with $|D|>2$ we define $\A_k$ to be the set of $k$-digit numbers whose base-$b$ representations only contain elements of $D$, more formally $\A_k$ is the set of numbers of the form
$$\sum_{i=0}^k a_i b^i$$
with $a_i\in D$.  We henceforth call sets $\A_k$ of this form, with general $k$ and $D$, {\it digitally restricted sets}, and similarly we call their elements digitally restricted numbers..  Let $\A$ be the union over $k$ of $A_k$, i.e. the set of digitally restricted integers with any number of digits. Let $E$ be the set of even $n$ which are not the sum of two primes and let $E(X)=E\cap\{1,...,X\}$. 
\begin{theorem}\label{diggold}  
There exists some $\delta=\delta(D,b)$ such that
$$|E(X)\cap \A| \ll |\A|^{1-\delta}.$$
\end{theorem}
To see the relevance of this result, we look at historical results on the Goldbach conjecture. First conjectured by Goldbach in 1742, little progress occurred until the work of Vinogradov \cite{Vinogradov1937} in 1937, who showed that all sufficiently large odd numbers are the sum of three primes. Heilbronn\cite{Heilbronn} noted in his review that one could use the same method to prove almost all even numbers the sum of two primes, and within a year Estermann \cite{Estermann}, van der Corput\footnote{Although van der Corput's paper shows up in the 1936 edition of Acta Arithmetica, it was not submitted until 1937} \cite{Corput} and Chudakov\cite{Chudakov37} all independently did so. The initial bound of $|E(X)| \ll X(\log X)^{-c}$ has since been improved, most notably by Montgomery and Vaughan in \cite{MVaughan} to $|E(X)|\ll X^{1-\delta}$ for some $\delta>0$, and the current best $\delta$ which can be found in the literature is $\delta=0.28$ by Pintz in \cite{Pintz72}. When considering digitally restricted sets with few allowed digits, there are significantly less than $X^{0.72}$ numbers with restricted digits less than $X$, and thus Theorem \ref{diggold} is nontrivial.

Another motivation for our result is previous study of how many exceptions lie in other specific thin sets. This was first studied by Perelli in \cite{Perelli}, who showed that for any polynomial $\phi$ of degree $k$ with integer coefficients the values of $2\phi(n)$ obey the Goldbach conjecture for all but 
$N(\log N)^{-A}$ many $n<N$. This estimate was improved by Br\"udern, Kawada, and Wooley in \cite{BKW} who lowered the bound on the exceptions to $N^{1-c/k}$ for some positive absolute constant $c$. 

Also motivating our result is work studying the set $\A$. The analytic structure of $\A$ has led to many results. Particularly notable results include Maynard\cite{Maynard} who proved that $\A$ has infinitely many primes provided that the digit set $D$ omits sufficiently few digits, and Mauduit and Rivat\cite{MR} who proved that sums of digits of primes are well-distributed in residue classes.

\section{Preliminaries}
We first introduce notation. As is common in analytic number theory, let $e(\alpha):=e^{2\pi i\alpha}$ and let the symbol $f(x)\ll g(x)$ mean that there exists a positive constant $c$ such that $|f(x)|\leq cg(x)$ after restricting to an appropriate domain. Whenever we uses the letter $\ep$, we mean that the statement is true for every $\ep>0$. Throughout this paper, the letter $p$ always denotes primes. Take $k$, the number of digits, to be sufficiently large, and let $X=b^k$, so that $\A_k$ has only values less than $X$.

When discussing digitally restricted sets we need to deal with the technicality of leading zeroes. We want to study all digitally restricted numbers less than $X$, of any number of digits, but when $0\not\in D$ we see that $\A_k$ only contains numbers with precisely $k$ digits and when $j< k$ we have that $\A_j\cap\A_{k}=\emptyset$, in contrast to the case where $0\in D$, in which a $j$-digit number is also a $k$-digit number with $k-j$ leading zeroes, and $\A_j\subset\A_k$. When $0\in D$ we may show that $|\A_k\cap E(X)|$ is small to obtain the desired result, but when $0\not\in D$ we note that $|\A\cap E(X)| = \sum_{j\leq k}|\A_k\cap E(X)|$, so proving the result for $\A_k$ will prove it for $\A(X)$ in either case, and we will henceforth deal only with $\A_k$.

Now we define the generating function for $\A_k$ to be
$$f(\alpha) := \sum_{n\in\A_k} e(n\alpha).$$
The generating function for primes on the other hand will require a parameter, let $\delta_0$ be a sufficiently small constant, which may depend on $b$, to be chosen later. We can now define 
$$S(\alpha) := \sum_{X^{6\delta_0}<p<X} e(p\alpha)\log p.$$
For any $\grB\subset [0,1]$ and any even $n<X$, we now define
$$r(n,\grB):=\int_\grB S(\alpha)^2e(-n\alpha)\de\alpha.$$ By orthogonality, we know that $r(n,[0,1])$ counts solutions to $p_1+p_2=n$ with weight $\log p_1\log p_2$. To evaluate this integral we apply the Hardy-Littlewood circle method. Define the major arcs to be
\begin{equation}
\M_{\delta_0} := \bigcup_{0<q<X} \{\alpha\in [0,1]:\norm{q\alpha}<X^{6\delta_0-1}\} \label{MajDef}
\end{equation}
and the minor arcs to be
\begin{equation}\m_{\delta_0} := [0,1]\setminus \M_{\delta_0}. \label{MinDef}\end{equation}
In section 3 we establish Lemmas \ref{orthocount} and \ref{sievebound}, the facts about $f(\alpha)$ and $\A_k$ which we will use in the proof. In section 4 we combine section 3 with commonly known facts about $S$ to prove Theorem \ref{diggold}.

\section{Properties of digital restrictions}
In this section, we prove the properties of digitally restricted sets which we will need to show Theorem \ref{diggold}. The two results we need are Lemma \ref{lpbound}, a mean value estimate, and Lemma \ref{sievebound}, a bound on how many elements of $\A_k$ can be multiples of a given large value.
\begin{lemma}\label{lpbound} For all  $\delta>0$, there exists a  $t_0=t_0(b,\delta)$ such that whenever $t\geq t_0$ we have the bound
$$\int_0^1 |f|^t\de \alpha \ll |\A_k|^tX^{\delta-1}.$$
\end{lemma}
\begin{proof} 
We prove this by exploiting the independence of digits of values in $\A_k$.  Note that by orthogonality, if $t=2s$ is an even integer, this integral counts solutions to the equation
\begin{equation}\label{orthocount}
\sum_{i=1}^{s}(x_i-y_i)=0
\end{equation}
with $x_i,y_i\in \A_k$. We now set each digit of each value to its own variable. Let
$$x_i = \sum_{j=0}^k x_{i,j}b^j$$
and the same for $y_i$. We are now counting solutions to
\begin{equation}\label{orthocountdig}
\sum_{j=0}^k b^j\sum_{i=1}^{s}(x_{i,j}-y_{i,j})=0
\end{equation}
with $x_{i,j},y_{i,j}\in D$.

First, if every allowed digit shares a common factor then we know that every $x_{i,j}$ and $y_{i,j}$ will share that common factor, permitting us to divide everything by that common factor. Thus, we may assume that the allowed digits do not share a common factor. To count solutions to (\ref{orthocountdig}) we investigate what addition looks like base-$b$, and in particular we look at each digital place. Let 
$$\sj{j} : = \sum_{i=1}^s (x_{i,j}-y_{i,j})$$
and
$$\tj{j} := \sum_{l<j}b^l\sj{l}.$$
Now, we note that for (\ref{orthocountdig}) to hold, we require that it holds (mod $b$), and thus $\sigma_0\equiv 0\mod b$, or to be precise, the 1s digits must add up to 0 mod $b$. Further, we know that \eqref{orthocountdig} holds mod $b^j$ for any $j$, meaning that $\tj{j} \equiv -b^j \sj{j} \mod b^{j+1}$. After we have chosen the digits up to $j$ we will thus require that $\sj{j} \equiv n$ mod $b$ for some fixed $n$ which is determined by the digits up to $j$. 

No matter what the digits in the other places are, we will thus have at most $$\max_{n<b}\#\{(x_{0,j},x_{1,j},...,x_{s,j},y_{0,j},...,y_{s,j}):\sj{j}\equiv n \mod b\}$$ choices for how to assign the $j$th digits. By orthogonality, we can count these using the function
\begin{equation}\label{onedig}
u(n,b,D):=\frac{1}{b}\sum_{0\leq a<b} \left(\left|\sum_{d\in D}e(da/b)\right|^{2s}e(-na/b)\right).
\end{equation}
Noting that this does not depend on $j$, we thus have that an upper bound on the number of solutions to (\ref{orthocountdig}) is the product of the upper bounds on the sums of each digit, which is bounded by the maximum of $u(n,b,D)^k$ over all possible $n$. Finally, noting that by coprimality of acceptable digits we have for all nonzero $a$ that
$$\abs{\sum_{d\in D}e(da/b)}<|D|.$$
Thus by raising to a sufficiently large $s$ we can make the ratio as small as we like, and thus make the sum over all values of $a$ as small as we like (while remaining above 1), obtaining that for a sufficiently large $s$ we have
$$u(b;D)=\frac{1}{b}\sum_{0\leq a<b}\abs{\sum_{d\in D}e(da/b)}^{2s} \leq \frac{(1+\delta )}{b}|D|^{2s}\leq \frac{1}{b}|D|^{2s+\delta}.$$
Taking the product across all digital places thus yields the desired result.
\end{proof}
\begin{lemma}\label{sievebound}
For any $m$, we have:
$$\#\{n\in\A_k:m| n\}\leq 2b\frac{|\A_k|}{m^{\log|D|/\log b}}.$$
\end{lemma}
\begin{proof}
Note that for every integer of the shape $am$ in $\A_k$, by changing the last $\left\lfloor\frac{\log m}{\log b}\right\rfloor$ many digits we can obtain $|D|^{\left\lfloor\frac{\log m}{\log b}\right\rfloor} \geq b^{-1}m^{\frac{\log|D|}{\log b}}$ many other values in $\A_k$ which are less than $m$ away from $am$. Since any value in $\A_k$ can be less than $m$ away from at most two integers of the shape $am$, the proportion of values in $\A_k$ which are of the shape $am$ can be at most $$\frac{2b}{m^{\frac{\log|D|}{\log b}}}.$$ Thus we obtain that there are at most $\frac{2b|\A_k|}{m^{\log|D|/\log b}}$ distinct integers of the shape $am$ which can be achieved in $|\A_k|$. 
\end{proof}
Since the expected number of multiples of $m$ lying in $\A_k$ would be $|\A_k|/m$, it is natural to ask whether we can obtain this or possibly $|\A_k|/m^c$ for some absolute $c$ which does not depend on $b$ and $|D|$. Unfortunately, this is ruled out by $m$ which are multiples of powers of $b$. Taking $m=b^j$, we see that $n\in \A_k$ being a multiple of $m$ is equivalent to the last $j$ digits being 0, and that when 0 is an allowed digit there will be $N/N^{j/k}=N/m^{\log|D|/\log b}$ of these. For values $m$ which are coprime to $b$, a significant amount of research has been done, see \cite{RestrictedDigitsModM}, but this research only applies to small values of $m$ and not the large values we will need. The question of whether or not there exists a constant $c=c(b)$ such that for all $m$ coprime to $b$, all residue classes mod $m$ have an element with restricted digits less than $m^c$ is open.
\section{Exceptional sets in Goldbach's Conjecture}
In this section, following \cite{BKW}, we use what we have proven about numbers with restricted digits to prove Theorem \ref{diggold}. We now study $r(n,\grB)$ as we defined it in the outline. In particular, we study $r(n,\M)$ and $r(n,\m)$ as defined in \eqref{MajDef} and \eqref{MinDef} respectively. Our goal will be to show for almost all $n$ in $\A_k$, that $|r(n,\m)|\ll X^{1-\delta_1}$ and $r(n,\M)\gg X^{1-\delta_2}$ with $\delta_1<\delta_2$. First, we handle the minor arcs.
\begin{lemma}\label{minarc}
There exists $\delta_1=\delta_1(b,D)>0$ such that for all sufficiently large $X$, there exists a set $B\subset \A_k$ with $|B|\ll |\A_k|X^{-\delta_1}$ and for all even $n$ in $\A_k\setminus B$ we have
$$|r(n,\m)|\ll |\A_k|X^{-\delta_1}$$
\end{lemma}
\begin{proof}
 Note that by symmetry of $\m$ we have that $r(n,\m)$ is real and set $\eta(n):=\mathrm{sign}(r(n,\m))$. Now, we let 
$$K(\alpha):= \sum\limits_{\substack{0<n<X\\n\in\A}} \eta(n)e(-n\alpha)$$
and rewrite the expression we are trying to bound as
$$\sum_{n\in \A_k} |r(n,\m)| = \int_\m S(\alpha)^2 K(-\alpha)\de\alpha.$$
H\"older's inequality yields that this is at most
$$\left(\sup_{\alpha\in\m}|S(\alpha)|\right)^{2/t}\left(\int_0^1|S(\alpha)|^2\de\alpha\right)^{1-1/t}\left(\int_0^1|K(-\alpha)|^t\de\alpha\right)^{1/t}.$$
and we can now estimate each term individually. 

First, by \cite{VaughanEstimate}, we have
$$\sup_{\alpha\in\m}|S(\alpha)| \ll X^{1-3\delta_0}(\log X)^4.$$
By applying orthogonality we trivially have
$$\int_0^1 |S(\alpha)|^2\de\alpha \ll \sum_{P<p<X}(\log p)^2 \ll X\log X.$$

Now we estimate the contribution of $K$. Observe that for integer values of $t$ we know that the mean value $\int_0^1 |K(\alpha)|^{2t}\de\alpha$ counts solutions with $x_i,y_i\in\A_k$ to the additive equation
$$\sumstack{i\leq t}(x_i-y_i) = 0,$$
with each solution being weighted with $\pm 1$. Thus it is clearly less than or equal to the unweighted version, namely $\int_0^1 |f(\alpha)|^{2t}\de\alpha$, which when we take $t$ sufficiently large is itself $O(|\A_k|^{2t}X^{\delta_0-1})$ by Lemma \ref{lpbound}.

Thus in total we have
\begin{align*}
\int_0^1S(\alpha)^2K(-\alpha)\de\alpha &\ll |\A_k|X^{1-1/t+(2-6\delta_0)/t+(-1+\delta)/t}(\log X)^{8/t+1-1/t}\\
&=|\A_k|X^{1-5\delta_0/t}(\log X)^{1+7/t}
\end{align*}
implying that
$$\sumstack{n<X\\n\in\A}|r(n,\m)|\ll NX^{1-2\delta_2}$$
for some $\delta_2>0$. Now using the trivial bound, that 
$$|\{n\in\A_k:|r(n,\m)|>X^{1-\delta_1}\}|\leq X^{\delta_1-1}\sum_{n\in |\A_k|}|r(n,\m)|,$$ 
we obtain the desired result. We note that we just bounded the number of values in $\A$ with a large minor arc, and the number of even values in $\A$ with a large minor arc is clearly even less. \end{proof}
Now we deal with the major arcs.
\begin{lemma}
Given $Y$ such that  $1\leq Y\leq |\A_k|^{\delta_0}$, we have that 
$$r(n,\M) \gg XY^{-1/2}(\log X)^{-1}$$
for all even $n\in\A_k(X)\setminus \A_k(X/10)$ except for at most $O(|\A_k|^{1+\ep}Y^{-\delta_3})$ many $n$.
\end{lemma}
\begin{proof}This follows immediately from the proof of Lemma 2 in \cite{BKW}. Note that the only step in the proof which depends on the set being counted was between equations (18) and (19) of the latter paper, where numbers $n$ such that $(n,\tilde{r})>Y$ were discarded. We replace that with
$$\sumstack{d|\tilde{r}\\ d>Y}\sumstack{n\in \A_k\\d|n} 1\ll \sumstack{d|\tilde{r}\\ d>Y} (N/d^{\log|D|/\log b}) \ll  N^{1+\ep}Y^{-\log|D|/\log b}.$$
If we now discard $O(N^{1+\ep}Y^{-\log|D|/\log b})$ many numbers, the rest of the proof follows. We note that since other values were thrown out in the rest of the proof, we cannot take $\delta_3=1-\log|D|/\log b$, and we leave it as the indeterminate $\delta_3$.\\

\end{proof}
Now that we have both the major arcs and the minor arcs, we are equipped to prove Theorem \ref{diggold}.  Taking $Y$ to be $X^{\delta_1/3}$, we obtain that there is some set $B$ such that for all $n\in B$, $r(n,\m)=o(r(n,\M))$, and 
$$|\A\setminus B| \ll|\A_k|^{1+\ep}X^{-(\delta_1\log|D|)/(3\log|b|)}+|\A_k|X^{-\delta_1}+|\A_k|^{1+\ep}X^{-\delta_1\delta_3/3}.$$ Recalling that the number of solutions to \eqref{Gold} is 
$$r(n,[0,1])=r(n,\m)+r(n,\M) \gg r(n,\M) > 0$$ 
for all other than those exceptions, we thus have that the Goldbach conjecture holds for all other values in $\A_k$.
\section{Acknowledgements}
The author's work is supported by NSF Grant DMS-2001549. The author would like to thank Trevor Wooley for considerable assistance on this paper.
\providecommand{\bysame}{\leavevmode\hbox to3em{\hrulefill}\thinspace}
\providecommand{\MR}{\relax\ifhmode\unskip\space\fi MR }
\providecommand{\MRhref}[2]{%
  \href{http://www.ams.org/mathscinet-getitem?mr=#1}{#2}
}
\providecommand{\href}[2]{#2}

\end{document}